\documentclass{article}
\usepackage{amsthm}
\usepackage{amsbsy}
\usepackage{amssymb}
\usepackage{amsfonts}
\usepackage{amsmath}
\usepackage{mathrsfs}
\usepackage{amscd}
\usepackage[dvips]{lscape}
\usepackage{verbatim} 
\usepackage{enumerate}

\theoremstyle{definition}
\newtheorem{thm}{Theorem}[section]
\newtheorem{lem}[thm]{Lemma}
\newtheorem{prp}[thm]{Proposition}
\newtheorem{dfn}[thm]{Definition}
\newtheorem{cor}[thm]{Corollary}

\newtheorem{rmk}[thm]{Remark}

\newcommand{\beq}{\begin{equation}}
\newcommand{\eeq}{\end{equation}}
\newcommand{\beqr}{\begin{eqnarray*}}
\newcommand{\eeqr}{\end{eqnarray*}}
\newcommand{\bal}{\begin{align*}} 
\newcommand{\eal}{\end{align*}}
\newcommand{\bei}{\begin{itemize}}
\newcommand{\eei}{\end{itemize}}

\newcommand{\af}{\alpha}
\newcommand{\bt}{\beta}
\newcommand{\gm}{\gamma}
\newcommand{\dt}{\delta}
\newcommand{\ep}{\varepsilon}

\newcommand{\ld}{\lambda}

\newcommand{\om}{\omega}

\newcommand{\Z}{{\mathbb{Z}}}

\newcommand{\C}{{\mathbb{C}}}
\newcommand{\T}{{\mathbb{T
}}}
\newcommand{\N}{{\mathbb{N}}}

\newcommand{\Her}{\mathrm{Her}}

\newcommand{\id}{{\mathrm{id}}}

\newcommand{\diag}{{\mathrm{diag}}}

\newcommand{\Act}{{\mathrm{Act}}}
\newcommand{\Aut}{{\mathrm{Aut}}}
\newcommand{\Ad}{{\mathrm{Ad}}}

\newcommand{\Min}{\mathrm{Min}}
\newcommand{\Max}{\mathrm{Max}}
\newcommand{\M}{\mathrm{M}}
\newcommand{\V}{\mathrm{V}}
\newcommand{\W}{\mathrm{W}}
\newcommand{\TS}{\mathrm{T}}
\newcommand{\DF}{\mathrm{DF}}

\newcommand{\QED}{\rule{0.4em}{2ex}}

\newcommand{\ifo}{if and only if}

\newcommand{\js}{\mathcal{Z}}

\newcommand{\pwl}{\precsim_{\mathrm{p.w.}}}
\newcommand{\cra}{\curvearrowright}

\begin{document}
\title{The Tracial Rokhlin Property for Actions of Amenable Groups on C*-Algebras}
\author{Qingyun Wang}
\maketitle
\begin{abstract}
The tracial Rokhlin property for finite group actions on simple C*-algebras was introduced by Chris Phillips in \cite{Phillips2011} to study the structure of the crossed product. It is much more flexible than the Rokhlin property, but still produces good structural theorems (See \cite{Phillips2011}, \cite{Elliott2008}). It should be viewed as the C*-version of outness that has the closest relationship with outness of actions on von Neumann algebras, while the later has been well developed. (See \cite{Connes1975}, \cite{Jones1980}, \cite{Ocneanu1985} e.g.) The tracial Rokhlin property for actions of $\Z$ has been studied by many authors (See \cite{Kishimoto1996}, \cite{Kishimoto1998}, \cite{Lin2006}, \cite{Osaka2006}, \cite{Osaka2006b} e.g.). Matui and Sato gave a definition of the tracial Rokhlin property for actions of discrete amenable groups in \cite{Matui2013b} and \cite{Matui2014}. They studied both the structure of the crossed product and classification of actions. But their definition is (at least formally) stricter than  the usual definition for finite group actions or $\Z$ actions, and their results works only for a special class of amenable groups.\\
In this paper, we present a definition of tracial Rokhlin property for (cocycle) actions of countable discrete amenable groups on simple C*-algebras, which generalize Matui and Sato's definition. We show that generic examples, like Bernoulli shift on the tensor product of copies of the Jiang-Su algebra, has the weak tracial Rokhlin property, while it was shown in \cite{Hirshberg2014} that such action does not have finite Rokhlin dimension. We show that forming crossed product from actions with the tracial Rokhlin property preserves the class of C*-algebras with real rank zero, stable rank one and has strict comparison for projections, generalizing the structural results in \cite{Osaka2006b}. We use the same idea of the proof with significant simplification. In another joint paper with Chris Phillips and Joav Orovitz, we shall show that pureness and $\js$-stability could be preserved by crossed product of actions with the weak tracial Rokhlin property. The combination of these results yields application in classification program, which will be discussed in the aforementioned paper. These results indicate that we have the right definition of tracial Rokhlin property for actions of general countable discrete amenable groups.
\end{abstract}
\section{Preliminary}\label{preliminaries}
Let $A$ be a C*-algebra in the following.\\
For $a,b\in A$, we mean by $[a,b]$ the commutator $ab-ba$. \\
For $\ep>0$, we write $a=_{\ep} b$ to mean $\|a-b\|<\ep$. For $B\subset A$, we write $a\in_{\ep} B$ if there is some $b\in B$ such that $a=_{\ep} b$.\\
If $h$ is a real function, then $h_+$ is the function defined by $h_+(t)=\Max\{0,h(t)\}$. If $a\in A$ is self-adjoint, then $a_+=\iota_+(a)$, where $\iota$ is the identity function.\\
The set of tracial states on $A$ is denoted by $\TS(A)$. For $a\in A$ and $\tau\in \TS(A)$, we define
\begin{equation*}
\|a\|_{2,\tau}=\|\tau(a^*a)^{1/2}\|,\hspace{1cm} \|a\|_2=\sup_{\tau\in \TS(A)}\|a\|_{2,\tau}.
\end{equation*}
If $\TS(A)$ is non-empty, then $\|\cdot\|_2$ is a semi-norm. For $\tau\in T(A)$, we let $\pi_\tau, H_\tau$ denote the GNS representation of $A$ associated with $\tau$. The dimension function $d_{\tau}$ associated with $\tau$ is given by
\begin{equation*}
d_{\tau}(a)=\lim_{n\rightarrow \infty}\tau(a^{1/n}),
\end{equation*}
for a positive element $a\in A$.\\
We write $\V(A)$ for the Murray-von Neumann semigroup and $\W(A)$ for the Cuntz semigroup. (See section 2 of \cite{Brown2008} for an introduction to the Cuntz semigroup). The space of states on $\W(A)$ is denoted by $\DF(A)$, where $\DF$ stands for dimension functions. For any $\tau\in \TS(A)$, $d_{\tau}$ give rise to a lower semicontinuous dimension functions on $A$.\\
Let $\om\in \bt\N\backslash \N$ be a free ultrafilter. Define
\begin{align*}
&c_\infty(A)=\{(a_n)\in\ell^{\infty}(\N,A)\,\vert\, \lim_{n\rightarrow \infty}\|a_n\|=0\}, \quad A^{\infty}=\ell^{\infty}(\N,A)/c_\infty(A);\\
&c_\om(A)=\{(a_n)\in\ell^{\infty}(\N,A)\,\vert\, \lim_{n\rightarrow \om}\|a_n\|=0\},\quad  A^\om=\ell^{\infty}(\N,A)/c_\om(A).
\end{align*}
Identify $A$ with the subalgebra of $A^{\infty}$ ($A^{\om}$) consists of constant sequences. Let
\begin{equation*}
A_{\infty}=A^{\infty}\cap A^{\prime}, \quad\quad A_{\om}=A^{\om}\cap A^{\prime}
\end{equation*}
and call them the central sequence algebras of $A$. \\
For a sequence $x=(x_i)_{i\in \N}$, define $\|x\|_{2,\om}=\lim_{n\rightarrow \om} \|x_n\|_2$. This defines a seminorm in $A^{\om}$. Let 
\beq
J_A=\{x\in A^{\om}\,\vert\, \|x\|_{2,\om}=0\}.
\eeq
Then $J_A$ is a well-defined two-sided close ideal in $A^{\om}$. \\
The cardinality of a set $F$ is written as $|F|$.
\begin{dfn}
Let $G$ be a countable discrete group.
\begin{enumerate}[(1)]
\item For a finite subset $K\in G$ and $\ep>0$, we say that a finite subset $T\subset G$ is $(K,\ep)$-invariant if $|T\cap\bigcap_{g\in F}gT|\geq (1-\ep)|T|$.
\item The group $G$ is amenable if for any finite subset $K\in G$ and $\ep>0$ there exists an $(K,\ep)$-invariant finite subset $T\in G$.
\end{enumerate}
\end{dfn}
Let $G$ be any discrete group, We write $\Act_G(A)$ to be the set of all actions $\af\colon G\rightarrow \mathrm{Aut}(A)$.

When $\af$ is an automorphism or an action of $A$, we can consider its natural extensions on $A^{\om}$ and $A_{\om}$. We shall denote it by the same symbol $\af$.\\
For $\af\in\mathrm{Aut}(A)$, we let 
\beq
\TS^{\af}(A)=\{\tau\in \TS(A)\,\vert\, \tau\circ\af=\tau\}.
\eeq


\begin{dfn}
Let $A$ be a unital C*-algebra and let $G$ be a discrete group.
\begin{enumerate}[(1)]
\item A pair $(\af, u)$ of a map $\af \colon G\rightarrow \Aut(A)$ and a map $u\colon G\times G \rightarrow U(A)$ is called a cocycle action of $G$ on $A$ if 
\begin{equation*}
\af_g\circ\af_h= \Ad u(g,h)\circ \af_{gh}
\end{equation*}
and
\begin{equation*}
u(g,h)u(gh,k)=\af_g(u(h,k))u(g,hk)
\end{equation*}
hold for any $g,h,k\in G$. We always assume $\af_{1}=\id$, $u(g,1)=u(1,g)=1$ for all $g\in G$. Notice that $\af$ give rise to a genuine action of $G$ on $A_{\om}$.
\item A cocycle action $(\af, u)$ is said to be outer if $\af_g$ is outer for every $g\in G$ except for the identity element.
\item Two cocycle actions $(\af, u)\colon G\cra A$ and $(\bt, v)\colon G\cra B$ are said to be cocycle conjugate if there exist a family of unitaries $(w_g)_{g\in G}$ in $B$ and an isomorphism $\theta\colon A\rightarrow B$ such that
\beq
\theta\circ\af_g\circ \theta^{-1}=\Ad w_g\circ \bt_g
\eeq
and
\beq
\theta(u(g,h))=w_g\bt_g(w_h)v(g,h)w_{gh}^*
\eeq
for every $g, h\in G$.
\end{enumerate}
\end{dfn}

\begin{dfn}
Let $(\af, u)\colon G\cra A$ be a cocycle action of a discrete group $G$ on a unital C*-algebra $A$. The (full) twisted crossed product $A\rtimes_{\af, u} G$ is the universal C*-algebra generated by $A$ and a family of unitaries $(\ld_g^{\af})_{g\in G}$ satisfying
\beq
\ld_g^{\af}\ld_h^{\af}=u(g,h)\ld_{gh}^{\af}\quad\quad\text{and}\quad\quad \ld_g^{\af}a(\ld_g^{\af})^*=\af_g(a)
\eeq
for all $g, h\in G$ and $a\in A$.
\end{dfn}
If two cocycle actions $(\af, u)\colon G\cra A$ and $(\bt, v)\colon G\cra B$ are cocycle conjugate, then $A\rtimes_{\af, u}G$ and $B\rtimes_{\bt, v}G$ are canonically isomorphic. \\
We introduce the following comparison for the convenience of studying tracial Rokhlin property. 
\begin{dfn}\label{pwl}
Let $f\in (A^{\om})_{+}$ and $a$ be an element of $A_{+}$. We say $f$ is pointwisely Cuntz subequivalent to $a$ and write $f\pwl a$, if $f$ has a representative $(f_n)_{n\in \N}\in \ell^{\infty}(\N, A)$, such that each $f_n$ is positive and $f_n\precsim a$ in $A$, for all $n\in\N$. 
\end{dfn}

\section{Equivalent definitions of the tracial Rokhlin property}
Through out this paper, we let $\om\in \bt\N\backslash \N$ to be some fixed free ultrafilter. We shall also assume that the groups acting on C*-algebras are countable, discrete and amenable.
\begin{dfn}\label{WTRP}
Let $A$ be a simple unital C*-algebra. Let $(\af, u)\colon G\cra A$ be a cocycle action. We say $\af$ has the tracial Rokhlin property, if for any finite subset $K$ of $G$, any $\ep>0$, and any $z\in A_{+}\backslash\{0\}$, there exist $(K,\ep)$-invariant finite subsets $T_1, T_2,\cdots,T_n$ and projections $\{e_i\,\vert\,  1\leq i\leq n\}\subset A_\om$ such that,
\begin{enumerate}[(1)]
\item $\af_{g}(e_i)\af_{h}(e_j)=0$, for any $g\in T_i, h\in T_j$ such that $g\neq h$ or $i\neq j$.
\item With $e=\sum_{g\in T_i,1\leq i\leq n}\af_{g}(e_i)$,\quad $1-e\pwl z$. (See Definition \ref{pwl})
\end{enumerate}
If the $e_i$'s in the above is weakened to be a positive contraction, then we say that $\af$ has the weak tracial Rokhlin property.
\end{dfn}

One can alternatively define the (weak) tracial Rokhlin property in terms of the original C*-algebra $A$ with approximate relations. 
\begin{prp}\label{edef}
Let $A$ be a simple separable unital C*-algebra. Let $(\af, u)\colon G\cra A$ be a cocycle action. Then $(\af, u)$ has the weak tracial Rokhlin property if and only if, for any finite subset $K$ of $G$, any $\ep_0>0$, and any non-zero positive element $z\in A$, there are $(K,\ep_0)$-invariant subset $T_1,\dots, T_n$ of $G$, such that for any finite subset $F$ of $A$, any $\ep_1>0$, there exists mutually orthogonal positive contractions $\{e_{g,i}\}_{g\in T_i, 1\leq i\leq n}$ with the following properties:
\begin{enumerate}[(1)]
\item $\|[e_{g,i}, f]\|<\ep_1$, for any $g\in T_i$ and any $f\in F$.
\item $\|\af_{hg^{-1}}(e_{g,i})-\lambda_{hg^{-1}, g}e_{h,i}\lambda_{hg^{-1},g}^*\|<\ep_1$, for any $g$ and $h$ in $T_i$.
\item With $e=\sum_{g\in T_i, 1\leq i\leq n}e_{g,i}$, we have $1-e\precsim z$.
\end{enumerate}
Furthermore, if $\af$ has the tracial Rokhlin property, then the positive contractions $e_g$ could always be chosen to be non-zero projections.
\end{prp}
\begin{rmk}
The above proposition shows that the definition of (weak) tracial Rokhlin property is independent of the choice of free ultrafilter. One can also use $A_\infty$ instead of $A_\om$ in the definition. For most of the results and proofs of this paper, it doesnot matter which one we use. However, one advantage of using a free ultrafilter instead of the sequence algebra is that, if $(M, \tau)$ is a tracial von Neumann algebra, then $M^{\om}=\ell^{\infty}(M)/c_{\om, \tau}(M)$ is again a von Neumann algebra, where $c_{\om, \tau}(M)=\{x\in M\,\vert\, \lim_{n\rightarrow \om} \tau(xx^*)^{1/2}=0\}$. While the analogue algebra $M^{\infty}=\ell^{\infty}(M)/c_{0, \tau}(M)$ is not a von Neumann algebra. We will make use of this fact in the proof of Proposition \ref{rl}.
\end{rmk}
When the C*-algebra has strict comparison, Cuntz comparison is equivalent to comparison by traces, which gives the following:
\begin{prp}\label{edefT}
Let $A$ be a simple unital C*-algebra with strict comparison. Let $(\af, u)\colon G\cra A$ be a cocycle action. Then $(\af, u)$ has the weak tracial Rokhlin property if and only if, for any finite subset $K$ of $G$, any $\ep_0, \ep_1>0$, there exists $(K,\ep_0)$-invariant subsets $T_1,\dots, T_n$ and positive contractions $e_1,\dots e_n\in A_{\om}$, such that
\begin{enumerate}[(1)]
\item $\af_g(e_i)\af_h(e_j)=0$, for $g\in T_i$ and $h\in T_j$ such that $g\neq h$ or $i\neq j$.
\item Let $e=\sum_{g\in T_i, 1\leq i\leq n}(\af_g(e_i))$, there is a representative $(e^{(n)})_{n\in\N}$ of $e$ such that
\beq
\lim_{n\rightarrow \om} \max_{\tau\in \TS(A)}d_\tau(1-e^{(n)})<\ep_1.
\eeq
\end{enumerate}
In case of tracial Rokhlin property, one could replace positive contraction by non-zero projection in the above statement.
\end{prp}

The above proposition lead to the following definition:
\begin{dfn}\label{WTRPT}
Let $(\af, u)\colon G\cra A$ be a cocycle action. Let $S\subset \TS(A)$. We say $(\af, u)$ has the (weak) tracial Rokhlin property with respect to $S$, if for any finite subset $K$ of $G$, any $\ep_0, \ep_1>0$, there exists $(K,\ep_0)$-invariant subsets $T_1,\dots, T_n$ and projections (positive contractions) $e_1,\dots e_n\in A_{\om}$, such that
\begin{enumerate}[(1)]
\item $\af_g(e_i)\af_h(e_j)=0$, for $g\in T_i$ and $h\in T_j$ such that $g\neq h$ or $i\neq j$.
\item Let $e=\sum_{g\in T_i, 1\leq i\leq n}(\af_g(e_i))$, there is a representative $(e^{(n)})_{n\in\N}$ of $e$ such that
\beq
\lim_{n\rightarrow \om} \max_{\tau\in S}d_\tau(1-e^{(n)})<\ep_1.
\eeq
\end{enumerate}
\end{dfn}

In the following we shall show that, if $A$ has tracial rank zero, then weak tracial Rokhlin property actually implies tracial Rokhlin property. The case $G=\Z$ has been  proved by Phillips and Osaka (See Theorem 2.14 of \cite{Osaka2006} and Proposition 1.3 of \cite{Phillips2012}). We need a lemma before we prove it.
\begin{lem}\label{mc}
Let $A$ be a C*-algebra and $B$ is finite dimensional subalgebra. Let $\{e_{ij}^{l}\}$ be the standard matrix units of $B$. Then for any $\ep>0$, there is a $\dt>0$, such that whenever a projection $p\in A$ satisfies $\|[p,e_{ij}^{l}]\|<\dt$ for all $i,j,l$, there is a projection $q$ in the relative commutant $A\cap B^{\prime}$, such that $\|p-q\|<\ep$.
\end{lem}
\begin{proof}
  Fix some $\ep>0$. Choose $\dt_0$ according to $\ep/2$ as in Lemma 2.5.10 of \cite{Lin2001a}. Choose $\dt_1$ according to $\dt_0$ as in Theorem 2.5.9 of \cite{Lin2001a} (It's easy to see that this lemma generalize to finite dimensional C*-algebras). Set $\dt=\dt_1/2$. Let $p\in A$ be a projection satisfying $\|[p,e_{ij}^{l}]\|<\dt$. Identify $C=pAp\oplus (1-p)A(1-p)$ as an subalgebra of $A$. Let 
\beq
a_{ij}^{l}=pe_{i,j}^{l}p+(1-p)e_{i,j}^{l}(1-p)\in C.
\eeq
Then $\|a_{i,j}^{l}-e_{i,j}^{l}\|<\dt_1$. Hence by Theorem 2.5.9 of \cite{Lin2001a}, there are matrix units $\{f_{i,j}^{l}\}\subset C$ such that $\|f_{i,j}^{l}-e_{i,j}^{l}\|<\dt_0$. By Lemma 2.5.10 of \cite{Lin2001a}, there is a unitary $u\in A$ such that $uf_{i,j}^{l}u^*=e_{i,j}^{l}$ and $\|u-1\|<\ep/2$. Now let $q=upu^*$, then $\|q-p\|<\ep$. We shall show that $q$ commutes with $B$ by showing that $q$ commutes with each $e_{i,j}^{l}$. Since $\{f_{i,j}^{l}\}\subset C$, we have $f_{i,j}^{l}=(1-p)f_{i,j}^{l}(1-p)+pf_{i,j}^{l}p$, for any $i,j,l$. Hence
\beq
qe_{i,j}^{l}=upf_{i,j}^{l}u^*=upf_{i,j}^{l}pu^*=uf_{i,j}^{l}pu^*=e_{i,j}^{l}q
\eeq
\end{proof}


\begin{thm}\label{WTT}
Let $(\af, u)\colon G\cra A$ be a cocycle action with the weak tracial Rokhlin property. If $A$ is a simple C*-algebra tracial rank zero, then $(\af, u)$ actually has the tracial Rokhlin property.
\end{thm}
\begin{proof}
If $A$ has tracial rank zero, then $A$ is tracially approximately divisible (Theorem 5.4 of \cite{Lin2007}). If we define tracial $\js$-absorption (See Definition 2.1 of \cite{Hirshberg2013}) using finite-dimensional C*-algebras whose simple component has arbitrary large size instead of matrix algebras, tracial approximate divisibility will imply tracial $\js$-absorption. Theorem 3.3 of \cite{Hirshberg2013} will still hold, using essentially the same proof. Hence $A$ has strict comparison. (It turns out that in the simple case, the aforementioned definition of tracially $\js$-absorbing coincide with Definition 2.1 of \cite{Hirshberg2013}, although we do not need it for this paper). Let $K$ be a finite subset of $G$, let $\ep_0>0$ be given. By Proposition \ref{edefT}, there exists $(K,\ep_0)$-invariant subsets $T_1,\dots, T_n$ and positive contractions $e_1,\dots e_n\in A_{\om}$, such that
\begin{enumerate}[(1)]
\item $\af_g(e_i)\af_h(e_j)=0$, for $g\in T_i$ and $h\in T_j$ such that $g\neq h$ or $i\neq j$.
\item Let $e=\sum_{g\in T_i, 1\leq i\leq n}(\af_g(e_i))$, there is a representative $(e^{(n)})_{n\in\N}$ of $e$ such that
\beq
\lim_{n\rightarrow \om} \max_{\tau\in \TS(A)}d_\tau(1-e^{(n)})<\ep_1/3.
\eeq
\end{enumerate}
The rest of the proof amounts to perturb the $e_i$'s to projections with the desired properties.\\
Let $F$ be a finite subset of $A$. Let $\eta>0$ be arbitrary. Set $M=|T_1|+\cdots+|T_n|$. Choose $\dt$ such that
\beq
\dt=\ep_1/3(M+1).
\eeq
 Since $A$ has tracial rank zero, there is a finite dimensional subalgebra $B\subset A$ with $1_B=p$, such that
\begin{enumerate}[(1)]
\item $\|[p,a]\|<\eta$, for any $a\in F$.
\item $pap\in_{\eta} B$.
\item $\tau(1-p)<\dt$, for any $\tau\in \TS(A)$.
\end{enumerate}
Consider $C=A^{\om}\cap B^{\prime}\supset A_{\om}$. The sequence algebra of a real rank zero C*-algebra is again real rank zero. As a consequence of Lemma \ref{mc}, we get $C=(A\cap B^{\prime})^{\om}$. In particular, the C*-algebras $C$ and $\{\overline{pe_iCe_ip}\}_{1\leq i\leq n}$ have real rank zero. Note here we regard $p\in A$ as constant sequence in $A^{\om}$, which commutes with each $e_i$. Choose a projection $q_i\in \overline{pe_iCe_ip}$ such that $\|q_ie_iq_i-pe_ip\|<\dt$. We first claim that  $\af_g(q_i)\af_h(q_j)=0$, for $g\in T_i$ and $h\in T_j$ such that $g\neq h$ or $i\neq j$. To prove this, since $q_i\in \overline{pe_iCe_ip}$, it implies that for any $\gm>0$, we can find $d_i\in C$ such that $\|q_i-pe_id_ie_ip\|<\gm$. Hence
\begin{align*}
&\|\af_g(q_i)\af_h(q_j)\|\\
\leq &\|\af_g(q_i-pe_id_ie_ip)\af_h(q_j)\|+\|\af_g(pe_id_ie_ip)\af_h(q_j-pe_jd_je_jp)\|+\\
&\|\af_g(pe_id_ie_ip)\af_h(pe_jd_je_jp)\|\\
\leq &\gm+(1+\gm)\gm+0=(2+\gm)\gm.
\end{align*}
Since $\gm$ is arbitrary, we have $\af_g(q_i)\af_h(q_j)=0$.\\
Let $q=\sum_{g\in T_i, 1\leq i\leq n}\af_g(q_i)$. Let $(q^{(m)})_{m\in \N}$ be a representative of $q$ such that each $q^{(m)}$ is a projection. We can estimate:
\begin{align*}
&\lim_{m\rightarrow \om} \max_{\tau\in \TS(A)} \{\tau(1-q^{(m)})\}\\
\leq &\lim_{m\rightarrow \om} \max_{\tau\in \TS(A)} \{\tau(1-q^{(m)}e^{(m)}q^{(m)})\}\\
\leq &\lim_{m\rightarrow \om} \max_{\tau\in \TS(A)} \{\tau(1-pe^{(m)}p)+\|pe^{(m)}p-q^{(m)}e^{(m)}q^{(m)}\|\}\\
\leq &\lim_{m\rightarrow \om} \max_{\tau\in \TS(A)} \{d_\tau(1-e^{(m)})+\tau(1-p)+\|pep-qeq\|\}\\
\leq& \ep_1/3+\dt+M\dt<\ep_1.
\end{align*}
 
Finally, for any $a\in F$, find $b\in B$ such that $\|pap-b\|<\eta$. Then
\begin{align*}
\|q_ia-aq_i\|&=\|pq_ipa-apq_ip\|\\
&\leq \|pq_ipa-pq_ipap\|+\|pq_ipap-papq_ip\|+\|papq_ip-apq_ip\|\\
&\leq \|pq_ib-bq_ip\|+2\eta+\eta+\eta=4\eta.
\end{align*}
Now we choose an increasing sequence of finite subsets $\{F_k\}$ whose union is dense in $A$. Let $\eta=1/k$, we can get a sequence of projections $\{q_{i,k}\}_{1\leq i\leq n, k\in\N}$ in $A^{\om}$ satisfing:
\begin{enumerate}[(1)]
\item $\|q_{i,k}a-aq_{i,k}\|\leq 4/k$, for any $a\in F_k$.
\item $\af_g(q_{i,k})\af_h(q_{j,k})=0$, for any $g\in T_i$ and $h\in T_j$ such that $g\neq h$ or $i\neq j$.
\item Let $q_k=\sum_{g\in T_i, 1\leq i\leq n}\af_g(q_{i,k})$. There is a representative  $(q^{(m)}_k)_{m\in \N}$ of $q_k$, such that
\beq
\lim_{m\rightarrow \om} \max_{\tau\in \TS(A)}\tau(1-q^{(m)}_k)<\ep_1.
\eeq
\end{enumerate}
 We can then use Cantor's diagonal argument to select projections $p_i$ in $A_{\om}$ satisfying the conditions in Proposition \ref{edefT}, therefore $\af$ has the tracial Rokhlin property.
\end{proof}

\section{Examples of actions with the weak tracial Rokhlin property}
\begin{dfn}
Let $G$ be a discrete group, let $A$ and $B$ be unital C*-algebras and let $\af\colon G\cra A$ and $\bt\colon G\cra B$ be actions of $G$ on $A$ and $B$. We say $B$ admits an approximate equivariant central unital homomorphism from $A$ if there is a sequence of unital completely positive maps $\phi_i\colon A\rightarrow B$ such that for any $a, a_1\in A$ and $b\in B$, we have
\begin{enumerate}[(1)]
\item $\lim_{i\rightarrow \infty} \phi_i(a)\phi_i(a_1)-\phi_i(aa_1)=0.$
\item $\lim_{i\rightarrow \infty} \phi_i(a)b-b\phi_i(a)=0.$
\item $\lim_{i\rightarrow \infty} \phi_i(\af_g(a))-\bt_g(\phi_i(a))=0.$
\end{enumerate}
\end{dfn}
It's immediate from the definition that an approximate equivariant central unital homomorphism from $A$ to $B$ induces an equivariant unital homomorphism from $A$ to the central sequence algebra $B_{\om}$.

\begin{thm}\label{bshift}
Let $\af\in \Act_G(A)$, where $G$ is amenable. Let $X$ be a compact metrizable space with a Borel probability measure $\mu$. Let $\bt\colon G\cra (X,\mu)$ be a free and measure-preserving action which is also a topological action (acts on $X$ by homeomorphisms). It induces an action on $C(X)$. Let $\tau$ be a tracial state on $A$. Suppose that there are approximate equivariant central unital homomorphisms $\iota_i\colon C(X)\rightarrow A$ with $\mu_i$ the measure induced by $\tau\circ\iota_i$. If $\mu$ is the $\om$-limit of $(\mu_i)_{i\in \N}$, then $\af$ has the weak tracial Rokhlin property with respect to $\tau$. Furthermore, if $X$ is totally disconnected, then $\af$ has the tracial Rokhlin property. 
\end{thm}
\begin{proof}
Let $K\in G$ be a finite subset, let $\ep_0, \ep_1>0$. Since $\bt\colon G\cra X$ is a free and measure preserving action, by Theorem 5 of \cite{Ornstein1987} (in page 59, section II.2) and the remark after the proof, there exists $(K,\ep_0)$-invariant subsets $T_1, T_2,\dots, T_n$ and measurable subsets $B_1, \dots, B_n$ such that:
\begin{enumerate}[(1)]
\item $gB_i$ and $hB_j$ are disjoint, for $g\in T_i$ and $h\in T_j$ such that $g\neq h$ or $i\neq j$.
\item $\mu(X\backslash \bigcup_{g\in T_i, 1\leq i\leq n}gB_i)<\ep_1$.
\end{enumerate}
Any finite measure on a compact metrizable space is regular, without loss of generality we may assume that each $B_i$ is compact. Now we are going to construct open sets $U_i\supset B_i$ such that $g\overline{U_i}$ and $h\overline{U_j}$ are disjoint, for $g\in T_i$ and $h\in T_j$ with $g\neq h$ or $i\neq j$.\\
Since $X$ is a normal topological space, for any $i$ and any $g\in T_i$, we can inductively find an open set $V_{g,i}\supset gB_i$ such that $\overline{V_{g,i}}$ and $\overline{V_{h,j}}$ are disjoint, for any $g\in T_i$ and $h\in T_j$ such that $g\neq h$ or $i\neq j$. For each $i$, define
\beq
U_i=\bigcap_{g\in T_i} g^{-1}V_{g,i}.
\eeq
It's easy to see from our construction that $U_i$ satisfies the requirement mentioned before. Furthermore, if $X$ is totally disconnected and compact, we may choose $V_{g,i}$'s to be clopen sets. It's clear from the construction that $U_i$'s are clopen sets as well.\\
Repeat the above argument, replacing $B_i$ by $\overline{U_i}$, we can get open sets $W_i\supset \overline{U_i}$ such that $g\overline{W_i}$ and $h\overline{W_j}$ are disjoint, for any $g\in T_i, h\in T_j$ such that $(g,i)\neq (h,j)$.\\
Now by Urysohn's lemma, we can find continuous functions $f_i\colon X\rightarrow [0,1]$ which is $1$ on $\overline{U_i}$ and $0$ outside $W_i$. If $X$ is totally disconnected, we let $f_i$ to be the characteristic function on the clopen set $W_i$. Let $e_i=(\iota_k(f_i))_{k\in \N}$, for $1\leq i\leq n$. Now we can see that $\{e_i\}_{1\leq i\leq n}$ are positive contractions (or projections if $X$ is totally disconnected) in $A_{\om}$ such that
\begin{enumerate}[(1)]
\item $\af_g(e_i)$ and $\af_h(e_j)$ are disjoint, for $g\in T_i$ and $h\in T_j$ such that $g\neq h$ or $i\neq j$.
\item With $f=\sum_{g\in T_i, 1\leq i\leq n}\af_g(f_i)$, and $e^{(k)}=\iota_k(f)$, we see that $(e^{(k)})_{k\in \N}$ is a representative of $e=\sum_{g\in T_i, 1\leq i\leq n}e_i$ such that
\begin{align*}
&\lim_{k\rightarrow \om} d_\tau(1-e^{(k)})=\lim_{k\rightarrow \om} d_{\mu_i}(1-f)\\
=&\lim_{k\rightarrow \om}\mu_i(\{{x\in X\,\vert\, 1-f(x)\neq 0}\})\\
\leq& \lim_{k\rightarrow \om}\mu_i(X\backslash \bigcup_{g\in T_i, 1\leq i\leq n}gU_i)\\
\leq&  \mu(X\backslash \bigcup_{g\in T_i, 1\leq i\leq n}gU_i)<\ep_1.
\end{align*}
\end{enumerate}
The second equality follows from Proposition I.2.1 of \cite{Blackadar1982}. By Proposition \ref{edefT}, the action $\af$ has the weak tracial Rokhlin property with respect to $\tau$, and has the tracial Rokhlin property with respect to $\tau$ if $X$ is further assumed to be totally disconnected.
\end{proof}

Let $A$ be a separable unital C*-algebra and $G$ be a countable discrete group. Let $\otimes_G A$ be the minimal tensor product of countably many copies of $A$ indexed by the elements of $G$. The left multiplication of $G$ on itself induces an action on $\otimes_G A$ (permuting the indices), which we shall call the Bernoulli shift on $\otimes_G A$. We can generalize Corollary of \cite{Phillips2012} to actions of amenable groups.
\begin{prp}
Let $A$ be a unital C*-algebra. Let $\tau\in \TS(A)$ such that, with $\pi_{\tau}$ being the associated GNS representation, the von Neumann algebra $\pi_\tau(A)^{''}$ has no minimal projections. When $G$ is infinite and amenable, then the Bernoulli shift on $\otimes_G A$ has the weak tracial Rokhlin property with respect to $\tau$.
\end{prp}
\begin{proof}
By Proposition 2.8 of \cite{Phillips2012}, there is some $a\in A$ with $0\leq a\leq 1$ and such that the spectral measure $\mu_0$ on $[0,1]$, defined by $\int_0^1 f\,d\,{\mu_0}=\tau(f(a))$ for $f\in C([0,1])$, satisfies $\mu_0(\{t\})=0$ for all $t\in [0,1]$. By functional calculus, there is an unital embedding $\iota\colon C([0,1])\rightarrow A$ defined by $f\rightarrow f(a)$. Let $X=\prod_G [0,1]$, the product of countably many copies of $[0,1]$ indexed by elements of $G$. There is a natural isomorphism between $C(X)$ and $\otimes_GC([0,1])$. We use $f_1^{(g_1)}\otimes f_2^{(g_2)}\otimes \cdots\otimes f_n^{(g_n)}$ to mean the elementary tensor in $\otimes_G A$ which is $f_i$ in $g_i$-th tensor factor for $1\leq i\leq n$ and is $1=1_A$ in all other places. Let $\af$ be the Bernoulli shift on $\otimes_G C([0,1])$ determined by
\beq
 \af_g(f_1^{(g_1)}\otimes f_2^{(g_2)}\otimes \cdots\otimes f_n^{(g_n)})=f_1^{(gg_1)}\otimes f_2^{(gg_2)}\otimes \cdots\otimes f_n^{(gg_n)}
\eeq
Using the natural isomorphism between $C(X)$ and $\otimes_GC([0,1])$ and the duality between actions on $C(X)$ and actions on $X$, we get an induced action $\bt$ on $X$ defined by
\beq
\bt_g(x_1^{(g_1)}, x_2^{(g_2)},  \cdots, x_n^{(g_n)}, \cdots)=(x_1^{(gg_1)}, x_2^{(gg_2)},  \cdots, x_n^{(gg_n)},\cdots)
\eeq
Let $\mu$ be the product measure on $X$ induced by $\mu_0$. It's easy to see that $\bt$ is measure preserving. We now check that it is also free. Given $g\in G\backslash\{1\}$, let $S=\{x\in X\,\vert\, \bt_g(x)=x\}$, we have $S=\{(x_h)_{h\in G}\,\vert\, x_h=x_k, \forall h,k\in G\}$. Using Fubini's theorem plus the assumption that single 
 point set in $[0,1]$ has zero measure, we get $\mu(S)=0$.\\
 List the elements in $G$ by $h_1, h_2, \dots$, let $\phi_k\colon C(X)\rightarrow \otimes_G A$ to be the 'right index shift by $h_k$', determined by 
\beq
f_1^{(g_1)}\otimes f_2^{(g_2)}\otimes \cdots\otimes f_n^{(g_n)}\rightarrow f_1(a)^{(g_1h_k)}\otimes f_2(a)^{(g_2h_k)}\otimes \cdots\otimes f_n(a)^{(g_nh_k)}
\eeq

We can  see that $\{\phi_n\}_{n\in \N}$ is a sequence of equivariant unital homomorphisms, we now check that it is approximately central. Let $f\in C(X)$ and $b\in \otimes_G(A)$. Without loss of generality we may assume that $f,b$ are elementary tensors:
\beq
f=f_1^{(g_1)}\otimes f_2^{(g_2)}\otimes \cdots\otimes f_n^{(g_n)}, \quad\quad b=b_1^{(h_1)}\otimes b_2^{(h_2)}\otimes \cdots\otimes b_n^{(h_n)}.
\eeq
There are only finitely many $g\in G$ such that $g_ig=h_j$ for some $1\leq i\leq j\leq n$, hence $lim_{k\rightarrow \infty} \phi_k(a)b-b\phi_k(a)=0$.\\
By Theorem \ref{bshift}, the action $\af$ has the weak tracial Rokhlin property with respect to $\tau$.
\end{proof}

In particular, for the Jiang-Su algebra $\js$, there is a central embedding of $C([0,1])$ such that the unique trace $\tau$ on $\js$ induces the Lebesgue measure on $[0,1]$. Hence we have:
\begin{cor}\label{bfonz}
If $G$ is countable, discrete and amenable, then the Bernoulli shift on $\otimes_G \js\cong \js$ has the weak tracial Rokhlin property.
\end{cor}


A cocycle action $(\af, u)\colon G\cra A$ is called strongly outer, if and only if for any $g\neq 1$ and any $\tau\in \TS^{\af_g}(A)$, the weak extension of $\af_g$ on $\pi_{\tau}(A)
''$ is not weakly inner.\\

\begin{prp}\label{so}
Let $G$ be a countable discrete amenable group, let $A$ be a unital simple infinite dimensional C*-algebra, let $(\af, u)\colon G\cra A$ be an action with the weak tracial Rokhlin property. Suppose that the tracial state space $\TS(A)$ has finitely many extreme points. Then $\af$ is strongly outer. 
\end{prp}
\begin{proof}
Let $1\neq g\in G$ be given, and let $\tau$ be an $\af_g$-invariant trace. Let $E\colon A\rtimes_{\af_g}\Z\rightarrow A$ be the conditional expectation determined by $E(a_n\ld_g^n)=a_0$, where $a_n\in A$ and $\ld_g$ is the canonical unitary in $A\rtimes_{\af_g}\Z$ implementing the action. We will show that, for any trace $\Phi\in \TS(A\rtimes_{\af_g}\Z)$, we have $\Phi(a\ld_g)=0$. If this is done, then the proof of Lemma 4.4 of \cite{Kishimoto1996} shows that $\af_g$ is not weakly inner. \\

For any $\tau\in \TS(A)$ and $x=(x_n)_{n\in \N}$, let $\tau_\om(x)=\lim_{n\rightarrow \om}\tau(x_n)$, which is a trace on $A_\om$. Let $\ep>0$ be arbitrary. Since $\af$ has the weak tracial Rokhlin property, by Proposition \ref{edefT}, we can find a $(\{g\},\ep)$-invariant subsets $T_1,\dots, T_n$ of $G$, and positive contractions $e_1,\dots, e_n$ in $A_{\infty}$, such that:
\begin{enumerate}[(1)]
\item $\af_g(e_i)\af_h(e_j)=0$, for $h\in T_i$ and $k\in T_j$ such that $h\neq k$ or $i\neq j$.
\item Let $e=\sum_{h\in T_i, 1\leq i\leq n}(\af_h(e_i))$, there is a representative $(e^{(n)})_{n\in\N}$ of $e$ such that
\beq
\lim_{n\rightarrow \om} \max_{\tau\in \TS(A)}d_\tau(1-e^{(n)})<\ep.
\eeq
\end{enumerate}
Now let $\tau_1,\tau_2,\dots,\tau_k$ be the extreme tracial states of $A$. Identify $a\ld_g$ with the constant sequence in $(A\rtimes_{\af_g}\Z)^{\infty}$, without loss generality assume $\|a\|=1$. We have
\begin{align}
|\Phi_{\om}(a\ld_g)|\leq& |\Phi_{\om}(\sum_{h\in T_i, i}\af_h(e_i)a\ld_g)|+|\Phi_\om((1-\sum_{h\in T_i,i}\af_h(e_i))a\ld_g)|\\
\leq& |\Phi_{\om}(\sum_{h\in T_i\cap gT_i, i}\af_h(e_i)a\ld_g)|+|\Phi_{\om}(\sum_{h\in T_i\backslash gT_i, i}\af_h(e_i)a\ld_g)|\\
&+\Phi_{\om}(1-\sum_{h\in T_i,i}\af_h(e_i))\|a\ld_g\|\\
\leq& |\Phi_{\om}(\sum_{h\in T_i\cap gT_i,i}\af_h(e_i^{1/2})a\ld_g\af_{g^{-1}h}(e_i^{1/2}))|+\\
&\sum_{h\in T_i\backslash gT_i,i}\Phi_{\om}(\af_h(e_i)))\|a\ld_g\|+\ep\\
\label{e1}\leq& 0+|\sum_{1\leq j\leq k}\sum_{h\in T_i\backslash gT_i,i} \tau_{j, \om}(\af_h(e_i))|+\ep\\ 
\label{e2}\leq&\sum_{1\leq j\leq k}\sum_{h\in T_i,i}\frac{|T_i\backslash gT_i|}{|T_i|}\tau_{j, \om}((\af_h(e_i)) +\ep\\ 
\label{e3} \leq& \ep\sum_{1\leq j\leq k}\tau_{j, \om}(\sum_{h\in T_i,i}\af_h(e_i))+\ep\leq (k+1)\ep.  
\end{align}
The estimation in the second to last line used the fact $\sum_{1\leq j\leq k}\tau_j((\af_h(a))$ is independent of $h$, since $\tau\rightarrow \tau\circ\af_h$ permutes the set of extreme tracial states. Since $\ep$ is arbitrary, this shows that $\Phi_{\om}(a\ld_g)=0$, and therefore $\Phi(a\ld_g)=0$.
\end{proof}

\begin{cor}\label{it}
Let $\af\in\Act_G(A)$ be an action with the weak tracial Rokhlin property. Then the canonical embedding $A\rightarrow A\rtimes_\af G$ induces a bijection between $\TS^{\af}(A)$ and $\TS(A\rtimes_\af G)$. 
\end{cor}
\begin{proof}
Let $r$ be the map from $\TS(A\rtimes_\af G)$ to $\TS^{\af}(A)$ induced by the canonical embedding $A\rightarrow A\rtimes_\af G$. Let $s$ be the map from $\TS^{\af}(A)$ to $\TS(A\rtimes_\af G)$, defined by
\beq
s(\tau)(\sum a_g\ld_g)=\tau(a_1), \quad\quad \forall \tau\in \TS^{\af}(A).
\eeq
It's easy to check that $r\circ s$ is the identity map. To prove that $s\circ r$ is the identity map, it suffices to show that for any trace $\Phi$ in $\TS(A\rtimes_\af G)$, any $g\neq 1$, we have $\Phi(a\ld_g)=0$. We repeat the same argument as in Proposition \ref{so} except the last three inequalities (\ref{e1}), (\ref{e2}) and (\ref{e3}). Note that $\Phi$ is now a trace on $\TS(A\rtimes_\af G)$, not just a trace on $\TS(A\rtimes_{\af_g}\Z)$; and we dropped the assumption that $A$ has finitely many extremal tracial states. Let $\tau=r(\Phi)\in \TS^{\af}(A)$, adopt the same notations as in Proposition \ref{so}, we have the following estimation:
\begin{align*}
\sum_{h\in T_i\backslash gT_i,i}\Phi_{\om}(\af_h(e_i)))=&\sum_{h\in T_i\backslash gT_i,i}\tau_{\om}(\af_h(e_i))\\
=&\sum_{h\in T_i,i}\frac{|T_i\backslash gT_i|}{|T_i|}\tau_{\om}(\af_h(e_i)) +\ep\\
\leq& \ep\tau_{\om}(\sum_{h\in T_i,i}\af_h(e_i))\leq \ep.
\end{align*}
Hence $\Phi_\om(a\ld_g)<2\ep$. Since $\ep$ is arbitrary, we have $\Phi(a\ld_g)=0$.
\end{proof}

We can now re-establish a Rokhlin-type lemma for outer actions on the hyperfinite $\mathrm{II}_1$ factor $R$ (It was discussed in Chapter 6 of \cite{Ocneanu1985} for actions of more general von Neumann algebras. The formulation is slightly different: our projections are permuted by the action exactly, but do not sum up to exactly 1). Let $p_\om\colon R^{\om}\rightarrow R^{\om}/J_R$, where $J_R$ is the trace-kernel defined in Section \ref{preliminaries}.
\begin{prp}\label{rl}
Let $G$ be an countable, discrete and amenable group. Let $R$ be the hyperfinite $\mathrm{II}_1$ factor. Let $\af\colon G\cra R$ be any outer action. Then for any finite set $K\in G$ and $\ep, \ep_1>0$, there exist $(K, \ep)$-invariant sets $T_1, \dots, T_n$ in $G$ and projections $p_1, \dots p_n\in R^{\om}/J_R\cap p_\om(R)^{'}$ such that 
\begin{enumerate}[(1)]
\item $\af_g(p_i)\af_h(p_j)=0$, for $g\in T_i$ and $h\in T_j$ such that $g\neq h$ or $i\neq j$.
\item $\tau_\om(1-\sum_{g\in T_i, 1\leq i\leq n}\af_g(p_i))<\ep_1$.
\end{enumerate}
\end{prp}

\begin{proof}
Any two outer actions on the hyperfinite $\mathrm{II}_1$ are cocycle conjugate (Theorem 1.4 of \cite{Ocneanu1985}), hence we need only to check one of them. \\
Let $\js$ be the Jiang-Su algebra. Let $\af$ be the Bernoulli shift on $\otimes_G \js\cong \js$. Let $\tau$ be the unique tracial state on $\js$. Then $\pi_\tau(\js)^{''}$ is the hyperfinite $\mathrm{II}_1$ factor $R$. The induced action on $R$, which we still call it $\af$, is outer by Corollary \ref{bfonz} and Proposition \ref{so}.\\
Let $K\in G$ be any finite set. Let $\ep, \ep_1>0$. Since $\af$ has the weak tracial Rokhlin property, there exist $(K, \ep)$-invariant sets $T_1, \dots, T_n$ in $G$ and positive contractions $e_1, \dots e_n\in \js_{\om}$ such that 
\begin{enumerate}[(1)]
\item $\af_g(e_i)\af_h(e_j)=0$, for $g\in T_i$ and $h\in T_j$ such that $g\neq h$ or $i\neq j$.
\item Let $e=\sum_{g\in T_i, 1\leq i\leq n}(\af_g(e_i))$, there is a representative $(e^{(n)})_{n\in\N}$ of $e$ such that
\beq
\lim_{n\rightarrow \om} \{d_\tau(1-e^{(n)})\}<\ep_1.
\eeq
\end{enumerate}
Using semiprojectivity of direct sums of $C_0((0,1])$, we can lift $\{\af_g(e_i)\}_{g\in T_i, i}$ to mutually orthogonal positive contractions $\{e_{g,i}=(e_{g,i}^{(n)})_{g\in T_i, i}\}$. Let $\tilde{e}^{(n)}=\sum_{g\in T_i, i}(e_{g,i}^{(n)})$, set $\dt_n=\|\tilde{e}^{(n)}-e^{(n)}\|$. Then $lim_{n\rightarrow \om} \dt_n=0$. Let 
\beq 
h_{\dt_n}(x)=1-\frac{1}{1-\dt_n}(1-x-\dt_n)_+, \forall x\in [0,1].
\eeq
Note that $h_{\dt_n}\in C_0((0,1])$ tends to the identity function as $n\rightarrow \om$. Let $f_{g,i}^{(n)}=h_{\dt_n}(e_{g,i}^{(n)})$, set $f_{g,i}=(f_{g,i}^{(n)})_{n\in \N}\in \ell^{\infty}(\N, A)$. Then $f_{g,i}$ is a representative of $\af_g(e_i)$.  Let$ f^{(n)}=\sum_{g\in T_i, i}(f_{g,i}^{(n)})$. We have:
\begin{align*}
1-f^{(n)}=1-h_{\dt_n}(\tilde{e}^{(n)})\approx (1-\tilde{e}^{(n)}-\dt_n)_+\precsim 1-e^{(n)}
\end{align*}

The algebra $R^{\om}/J_R$ is again a von Neumann algebra. For each $i$, let $p_i$ be the support projection of $p_\om(e_i)\in R^{\om}/J_R$. Since multiplication is strongly continuous on bounded sets and $\js$ is strongly dense in $R$, we have $p_i\in R^{\om}/J_R\cap p_\om(R)^{'}$ and $\af_g(p_i)\af_h(p_j)=0$, for $g\in T_i$ and $h\in T_j$ such that $g\neq h$ or $i\neq j$.\\
Let $\tilde{p}_{g,i}^{(n)}$ be the support projection of $f_{g,i}^{(n)}$ and set $\tilde{p}_{g,i}=(\tilde{p}_{g,i}^{(n)})_{n\in \N}$. Since $p_{\om}$ is strongly continuous, we see $\tilde{p}_{g,i}$ is a lift of $\af_g(p_i)$. Let $p^{(n)}=\sum_{g\in T_i, i}(p_{g,i}^{(n)})$. Since $f^{(n)}p^{(n)}=f^{(n)}$, easy calculation shows that $(1-f^{(n)})(1-p^{(n)})=1-p^{(n)}$. If we let $q^{(n)}$ be the support projection of $1-f^{(n)}$, then $1-p^{(n)}\leq q^{(n)}$. Using the fact that $d_\tau(1-f^{(n)})=\tau(q^{(n)})$, we have
\beq
\lim_{n\rightarrow \om} \tau(1-p^{(n)})\leq \lim_{n\rightarrow \om} d_\tau(1-f^{(n)})<\ep_1.
\eeq
\end{proof}


\begin{thm}\label{so=trp}
Let $A$ be a unital, simple, separable, infinite dimensional C*-algebra with finitely many extremal tracial states. Suppose $A$ is either nuclear or has tracial rank zero. Let $G$ be a countable discrete amenable group. For a cocycle action $(\af, u)$ of $G$ on $A$, it is strongly outer if and only if it has the weak tracial Rokhlin property.
\end{thm}
\begin{proof}
If $A$ is either nuclear or has tracial rank zero, then every trace $\tau$ is uniformly amenable (See Definition 3.2.1, Theorem 4.2.1 and Proposition 4.5  of \cite{Brown2006}). By Theorem 3.2.2 of \cite{Brown2006}, The von Neumann algebra $\phi_\tau(A)^{''}$ is hyperfinite. Since $A$ is unital, simple, infinite dimensional, $\phi_\tau(A)^{''}$ is the hyperfinite $\mathrm{II}_1$ factor. Now we can see that the proof of Theorem 3.7 of \cite{Matui2014} could be generalized to actions of general discrete amenable groups. The only change we need to make is to replace property (Q) by the property in Proposition \ref{rl} and modify the estimations accordingly. 
\end{proof}

Another type of example comes from the product-type actions. We begin with the definition:
\begin{dfn}\label{ptact}
Let $A=\otimes_{i=1}^{\infty}\mathrm{B}(H_i), $ where $H_i$ is a finite dimensional Hilbert space for each $i$. An action $\af\in\Act_G(A)$ is called a {\emph{product-type action}} \ifo\ for each $i$, there exists a unitary representation $\pi_i\colon G\rightarrow \mathrm{B}(H_i)$, which induces an inner action $\af_i\colon g\mapsto \Ad(\pi_i(g))$, such that $\af=\otimes_{i=1}^{\infty}\af_i$.
\end{dfn}

\begin{dfn}
Let $\af\in \Act_G(A)$ be a product-type action on a UHF-algebra $A.$ A \emph{telescope} of the action is a choice of an infinite sequence of positive integers $1=n_1<n_2<\cdots$ and a re-expression of the action, so that $A=\otimes_{i=1}^{\infty} B(T_i)$ where $T_i=\otimes_{j=n_i}^{n_{i+1}-1}H_j$, and the action on $B(T_i)$ is $\otimes_{j=n_i}^{n_{i+1}-1}\af_j$.
\end{dfn}

\begin{thm}\label{main}
Let $\af \in \Act_G(A)$ be a product-type action where $G$ is countable, discrete and amenable. Let $H_i$, $\pi_i$ and $\af_i$ be defined as in Definition \ref{ptact}. Let $d_i$ be the dimension of $H_i$ and $\chi_i$ be the character of $\pi_i$. We will use the same notations if we do a telescope to the action. Define $\chi\colon G\mapsto \C$ to be the characteristic function on $1_{G}$. Then the action $\af$ has the tracial Rokhlin property \ifo\ there exists a telescope, such that for any $n\in \N$, the infinite product
\begin{equation}
{\displaystyle{\prod_{n\leq i <\infty}\frac{1}{d_i}\chi_i}}=\chi.
\end{equation}
\end{thm}
\begin{proof}
  Any UHF algebra has tracial rank zero and monotracial. By Theorem \ref{so=trp}, that $\af$ has the tracial Rokhlin property is equivalent to that $\af$ is strongly outer. In this case, $\af$ has the tracial Rokhlin property if and only if $\af\vert_H$ has tracial Rokhlin property for any cyclic subgroup $H\subset G$.\\
Let $\chi_{H,i}$ be the restriction of $\chi_i$ to the subgroup $H$, which is exactly the character of the restricted action $\pi_i\vert_H$. We observe that $\prod_{n\leq i <\infty}\frac{1}{d_i}\chi_i=\chi$ if and only if 
\begin{equation}
{\displaystyle{\prod_{n\leq i <\infty}\frac{1}{d_i}\chi_{H,i}}}=\chi, \forall\,\, \text{cyclic subgroup}\,\, H\subset G.
\end{equation}
Hence the theorem will be proved if we can show that it is true for any cyclic group $G$. If $G$ is finite, then it is proved in \cite{Wang2013}. If $G$ is infinite, Let $x$ be a generator and $U_i$ be the unitary in $B(H_i)$ such that $\pi_i(x)=\Ad U_i$. Let $S_{k,l}$ be a sequence consisting of eigenvalues of $\otimes_{i=k}^lU_i$, repeated as often as multiplicity indicates. Kishimoto has shown that in case of infinite cyclic group acting on UHF algebra, the tracial Rokhlin property coincide with the Rokhlin property (Theorem 1.3 of \cite{Kishimoto1995}). He also show in Lemma 5.2 of \cite{Kishimoto1995} that the product-type action $\af$ has the Rokhlin property if and only $\{S_{k,l}\}_{l=k}^{\infty}$ is uniformly distributed, for any $k\in \N$.
Now fix some $k\in \N$. For any sequence $S=(\lambda_1,\lambda_2\dots,\lambda_n)$ in $\T$, We let $\mu_S$ be the measure on $\T$ such that $\mu_S=\frac{1}{n}\sum_{i}\dt_{\lambda_i}$, where $\dt_{\lambda_i}$ is the Dirac measure concentrated at the point $\lambda_i\in \T$. By definition,  $\{S_{k,l}\}_{l=k}^{\infty}$ is uniformly distributed if and only if
\beq
\lim_{l\rightarrow \infty} \mu_{S_{k,l}}(f)=\int_{\T} f\,d\mu, \forall f\in C(\T),
\eeq
where $\mu$ is the normalized Haar measure.
Now it's not hard to see that
\beq
{\displaystyle{\prod_{k\leq i <l}\frac{1}{d_i}\chi_i(n)}}=\mu_(S_{k,l})(z^n), \forall n\in \Z,
\eeq
where $z^n\in C(\T)$ stands for the function $z\rightarrow z^n$.\\
Hence
\begin{equation}
{\displaystyle{\prod_{k\leq i <\infty}\frac{1}{d_i}\chi_i}}=\chi.
\end{equation}
is equivalent to 
\beq
\lim_{l\rightarrow \infty} \mu_{S_{k,l}}(z^n)=\dt(n,0)=\int_{\T} z^n\,d\mu, \forall n\in \Z
\eeq
And therefore further equivalent to that  $\{S_{k,l}\}_{l=k}^{\infty}$ being uniformly distributed, since any continuous function in $C(\T)$ can be uniformly approximated by finite linear combinations of the functions $z^n$.
\end{proof}

Another example comes from actions on non-commutative tori. Let $\theta$ be a non-degenerate anti-symmetric bicharacter on $\Z^d$. We identify it with its matrix under the canonical basis of $\Z^d$. Then the associated non-commutative tori $A_{\theta}$ is simple, unital $\textbf{A}\T$ algebra with a unique trace. $A_{\theta}$ is generated by unitaries $\{U_x\,\vert\,x\in \Z^d\}$ subject to the relation

\beq
U_yU_x=exp({\pi i<x,\theta y>})U_{x+y}, \forall x,y\in \Z^d.
\eeq

For any $T\in \M_d(\Z)$, the map $U_x\rightarrow U_{Tx}$ give rises to an endomorphism $\af_T$ of $A_{\theta}$ if and only if $\frac{1}{2}(T^{t}\theta T-\theta)\in \M_d(\Z)$ (This relation is automatically satisfied for $d=2$). It is an automorphism if and only if $T$ is invertible. Let
\beq
G_{\theta}=\{T\in GL_n(\Z)\,\vert\, \frac{1}{2}(T^{t}\theta T-\theta)\in \M_d(\Z)\}
\eeq

\begin{prp}
Let $\theta$ be a non-degenerate anti-symmetric bicharacter on $\Z^d$. Let $G$ be any amenable subgroup of $G_{\theta}$. Then the action $\af\in \Act_G(A_{\theta})$, defined by $T\rightarrow \af_T$, is strongly outer, and hence has the tracial Rokhlin property. 
\end{prp}
\begin{proof}
Let $\tau$ be the unique trace state on $A_{\theta}$. By Lemma 5.10 of \cite{Echterhoff2010}, for each $T\in G\backslash \{e\}$, the automorphism $\af_T$ is not weakly inner. Hence $\af$ is strongly outer.
\end{proof}

If we can find one example of actions with (weak) tracial Rokhlin property, we can actually find lots of them by forming inner tensors. More specifically, we have the following:
\begin{prp}\label{tena}
Let $\af\in \Act_G(A)$ be an action with the weak tracial Rokhlin property and $\bt\in \Act_G(B)$ be arbitrary, where $A,B$ are both simple and unital. Then the inner tensor of these two actions $\gm=\af\otimes \bt \in \Act_G(A\otimes_{\min}B)$ has the weak tracial Rokhlin property. If $\af$ has the tracial Rokhlin property, then $\gm$ has the tracial Rokhlin property.
\end{prp}
\begin{proof}
Let $K\subset G$ be any finite subset and  $\ep>0$ be arbitrary. Since $\af$ has the weak tracial Rokhlin property, we can find $(K,\ep)$-invariant subsets $T_1,\dots, T_n$ of $G$ with the property stated in the definition of weak tracial Rokhlin property. Let $x\in A\otimes_{\min} B$ be a non-zero positive element.\\

We first show that there is an non-zero positive element $d\in A$ such that $d\otimes 1\precsim x$. By Kirchberg's slice lemma (Lemma 2.7 of \cite{Kirchberg1994} or Lemma 4.1.9 of \cite{Rordam2002}), there are non-zero positive elements $a\in A_+$ and $b\in B_+$ and some $z\in A\otimes_{\min} B$ such that $zz^*=a\otimes b$ and $z^*z\in \Her(x)$. This in particular shows that $a\otimes b\precsim x$. Since $B$ is simple and unital, we can find elements $s_1, s_2, \dots, s_n$ in $B$ such that $\sum_{i}s_ibs_i^*=1$. By Proposition 4.10 of \cite{Kirchberg2000} , we can find a non-zero positive contraction $d\in A$ such that $d^{\oplus n}\precsim a$. Hence
\beq
d\otimes 1=\sum_{i} (1\otimes s_i)(d\otimes b)(1\otimes s_i)^*\precsim (d\otimes b)^{\oplus n}\sim d^{\oplus n}\otimes b\precsim a\otimes b\precsim x. 
\eeq

Since $\af$ has the weak tracial Rokhlin property, there exist positive contractions $f_i\in A_{\om}$ such that:
\begin{enumerate}[(1)]
\item $\af_g(f_i)\af_h(f_j)=0$, for any $g\in T_i, h\in T_j$ such that $(g,i)\neq (h,j)$.
\item With $e=\sum_{g\in T_i, 1\leq i\leq n}\af_g(f_i)$, $1-e\pwl d$.
\end{enumerate}
Now consider the positive contractions $f_i\otimes 1$, it's clear that $f_i\otimes 1\in (A\otimes_{\min} B)_{\om}$ and:
\begin{enumerate}[(1)]
\item $\gm_g(f_i\otimes 1)\gm_h(f_j\otimes 1)=(\af_g(f_i)\af_h(f_j))\otimes 1=0$, for any $g\in T_i,h\in T_j$ such that $(g,i)\neq (h,j)$.
\item With $\tilde{e}=\sum_{g_i\in T_i, 1\leq i\leq n}\gm_g(f_i\otimes 1)$, we have
\beq
1-\tilde{e}\pwl d\otimes 1\precsim x.
\eeq
\end{enumerate}
Hence $\gm=\af\otimes \bt$ has the weak tracial Rokhlin property.\\
If $\af$ has the tracial Rokhlin property, then we can require $f_i$ to be non-zero projections, then $f_i\otimes 1$ are also projections, the above proof shows that $\gm$ has the tracial Rokhlin property. 
\end{proof}

\begin{rmk}
Let $G$ be any countable discrete amenable group, it admits at least one action on $\js$ with the weak tracial Rokhlin property (Corollary \ref{bfonz}). We then get lots of actions with the weak tracial Rokhlin property on any $\js$-stable C*-algebra $A$, by the above Proposition. Following the same argument as in \cite{Phillips2012}, we can actually show that the set of actions with the weak tracial Rokhlin property is $G_\dt$-dense in $\Act_G(A)$, where $\Act_G(A)$ is endowed with the topology of pointwise convergence. In particular, by Theorem \ref{WTT}, if $A$ is simple with tracial rank zero, then actions with the tracial Rokhlin property forms a $G_\dt$-dense subset of $\Act_G(A)$.
\end{rmk}


 

In the following two sections, we generalize results in \cite{Osaka2006b}. The idea is adapted from there.

\section{The Murray-von Neumann semigroup}
For a C*-algebra $A$, we let $\V(A)$ be the Murray-von Neumann semigroup of $A$. We say that $\V(A)$ has strict comparison if for any $p,q\in V(A)$, we have that $\tau(p)<\tau(q)$ for any $\tau\in \TS(A)$ implies $p\lesssim q$. Note that such C*-algebra is said to satisfy Blackadar's Second Fundamental Comparability Question, or that the order of projections is determined by traces in different literature.\\
We say $\V(A)$ is almost divisible, if for any $p\in \V(A)$ and any $n\in \N$, there is some $q\in \V(A)$ such that $nq\leq p \leq (n+1)q$.\\
Note that if $A$ is simple infinite dimensional with real rank zero, then $\V(A)$ is almost divisible, by Lemma 2.3 of \cite{Osaka2006b}.

\begin{lem} \label{large}
Let $A$ be a unital simple separable C*-algebra with Property (SP). Suppose that $\V(A)$ has strict comparison and is almost divisible. Let $(\af, u)\colon G\cra A$ be a cocycle action with the tracial Rokhlin property. Then for every finite subset $F\subset A\rtimes_{\af, u} G$, every $\ep>0$, and every nonzero $z\in (A\rtimes_{(\af,u)} G)_+$, there exists some finite subset $K$ of $G$ and $(K,\ep)$-invariant subsets $T_1, T_2, \dots, T_n$ of $G$,  projections $f_1, \dots, f_n\in A$, an embedding $\phi\colon \oplus_i M_{|T_i|}\otimes f_iAf_i\rightarrow A\rtimes_{\af, u} G$ whose image shall be called $D$ such that
\begin{enumerate}[(1)]
\item \label{i1} There is a $g_i\in T_i$ for each $i$ such that $\phi(e_{g_i,g_i}^{(i)}\otimes a)=\af_{g_i}(a)$, for any $a\in f_iAf_i$.
\item \label{i2} $\phi(e_{g,g}\otimes f_i)\in A$, for any $g\in T_i$ and $1\leq i\leq n$.
\item \label{i3} $\|\phi(e_{g,h}^{(i)}\otimes a)-\lambda_ga\lambda_h^*\|\leq \ep\|a\|$, for any $g,h\in T_i$ and $a\in f_iAf_i$.
\item \label{i4} Let $\tilde{T}_i=\bigcap_{g\in K} gT_i\cap T_i$. Let $p=\sum_{g\in \tilde{T}_i, 1\leq i\leq n}\phi(e_{g,g}^{(i)})$. We have
\beq
 pb\subset_{\ep} D\quad \text{and}\quad bp\subset_{\ep} D, \quad\text{for any}\,\, b\in F.
\eeq
\item \label{i6} With $p$ defined as in (\ref{i4}), $1-p\precsim z$.
\end{enumerate}
\end{lem}

\begin{proof}
We first choose two nonzero orthogonal positive elements $z_0, z_1\in A_+$ such that $z_0\oplus z_1\precsim z$ according to Lemma 5.1 of \cite{Hirshberg2013}. Since $A$ has property (SP), we could assume that $z_0$ and $z_1$ are projections. Let $\eta=\Min_{\tau\in T(A)} \tau(z_0)>0$. Let $\ep_0=\Min\{\frac{\eta}{2},\ep\}$. Without loss of generality assume that there is a symmetric finite set $K\subset G$ such that elements of $F$ are all of the form $\sum_{g\in K}a_g\lambda_g$, where $a_g$ are elements of $A$ and $\ld_g$ are the canonical unitaries implementing the action. \\
By Definition \ref{WTRP}, we can find $(K,\ep_0)$-invariant subsets $T_1, T_2,\cdots, T_n$ of $G$ and central projections $q_i\in A_{\infty}$ such that 
\begin{enumerate}
\item $\af_g(q_i)\af_h(q_j)=0$, for $g\in T_i$ and $h\in T_j$ such that $(g,i)\neq (h,j)$.
\item  $1-\sum_{g\in T}\af_g(q)\pwl z_1$. 
\end{enumerate}
For $1\leq i\leq n$, let $\{e_{g,h}^{(i)}\}$ be the standard matrix units of $M_{|T_i|}$. By the universal property of finite dimensional C*-algebras, there is an embedding 
\beq
\psi\colon \oplus_{1\leq i\leq n} M_{|T_i|}\rightarrow (A\rtimes_{\af,u} G)^{\infty}
\eeq
such that $\psi(e_{g,h}^{(i)})=\ld_gq_i\ld_h^*$.\\
Using semiprojectivity of $M$, we can lift $\psi$ to a sequence of embeddings $\psi_k \colon \oplus_i M_{|T_i|}\rightarrow (A\rtimes_{\af,u} G)$. We could further assume that $\psi_k(e_{g,g}^{(i)})\in A$, for $g\in T_i$ by standard perturbation argument (See Lemma 2.5.7 of \cite{Lin2001a}). \\
Now fix some $g_i\in T_i$ for each $i$. Let $q_{i,k}=\ld_{g_i}^*\psi_k(e_{g_i,g_i}^{(i)})\ld_{g_i}\in A$. We see that $(q_{i,k})_{k\in\N}$ is a representative of $q_i$. Let
\beq
 F_0=\{\af_g(a_h)\,\vert\, \sum_{h\in K}a_h\ld_h\in F, g\in \cup_i T_i\}\cup \{\af_k(u_{g,h})\,\vert\, g,h,k\in \cup_i(T_i\cup T_i^{-1})\}. 
\eeq
Let $L=\Max \{\|a\|\,\vert\, a\in F_0\}$. Define
\beq
\dt=\Min\{1/2, \,\,\frac{\ep}{|K|(\sum_i|T_i|)(L+5)},\,\, \ep/2\}
\eeq
We can find some large enough $k$ such that:
\begin{enumerate}[$(1^{\prime})$]
\item Let $e_g^{(i)}=\psi_k(e_{g,g}^{(i)})\in A$, we have $\|[e_g, a]\|<\dt$, for any $g\in T$ and any $a\in F_0$.
\item Let $f_i=q_{i,k}$, we have $\|\psi_k(e_{g,h}^{(i)})-\lambda_gf_i\lambda_h^*\|<\dt$, for any $g\in T$.
\item With $e=\sum_{g\in T_i, 1\leq i\leq k}e_g^{(i)}$, we have $1-e\precsim z_1$.
\end{enumerate}
The last condition comes from the fact that if two projections are close enough, then they are unitarily equivalent.\\
We now define an embedding $\phi\colon \oplus_iM_{|T_i|}\otimes f_iAf_i\rightarrow A\rtimes_{\af,u} G$ by
\beq
\phi(e_{g,h}^{(i)}\otimes a)=\psi_k(e_{g,g_i}^{(i)})\af_{g_i}(a)\psi_k(e_{g_i,h}^{(i)}).
\eeq
and extend linearly.\\
Let $D=\phi(\oplus_iM_{|T_i|}\otimes f_iAf_i)$ be the image of $\phi$.  Let $\tilde{T}_i=\cap_{g\in K} gT_i\bigcap T_i$ and 
\beq
p=\phi(\sum_{g\in \tilde{T}_i,1\leq i\leq n} (e_{g,g}^{(i)}\otimes f_i))=\sum_{g\in \tilde{T}_i, 1\leq i\leq n}e_g^{(i)}.
\eeq
We now verify the conditions required in this lemma. Condition (\ref{i1}) and (\ref{i2}) follows for the definition.\\ 
For condition (\ref{i3}), we have the following estimation:
\begin{align*}
\phi(e_{g,h}^{(i)}\otimes a) &=_{2\dt\|a\|} \ld_gf_i\ld_{g_i}^*\af_{g_i}(a)\ld_{g_i}f\ld_{h}^* \\
&= \ld_gf_iaf_i\ld_{h}^*=\ld_ga\ld_h^*.
\end{align*}
Hence $\|\phi(e_{g,h}^{(i)}\otimes a)-\ld_ga\ld_h^*\|\leq 2\dt\|a\|\leq \ep\|a\|$. \\
For condition (\ref{i4}), Let  $b=\sum_{h\in K}b_h\ld_h\in F$, we have:
\begin{align*}
pb&=\sum_{g\in \tilde{T}_i, 1\leq i\leq n, h\in K} e_g^{(i)}b_h\ld_h\\
&=_{\dt|K|(\sum_i |\tilde{T}_i|)L}\sum_{g\in \tilde{T}_i, 1\leq i\leq n, h\in K} \ld_gf_i\ld_g^*b_h\ld_h\\
&=\sum_{g\in \tilde{T}_i, 1\leq i\leq n, h\in K} \ld_gf_i\ld_{g^{-1}}u(g,g^{-1})b_h\ld_{g^{-1}}^*u(g^{-1},h)u(g^{-1}h, hg^{-1})\ld_{h^{-1}g}^*\\
&=_{\dt_1}\sum_{g\in \tilde{T}_i, 1\leq i\leq n, h\in K} \ld_gf_i\af_{g^{-1}}(u(g,g^{-1})b_h)u(g^{-1},h)u(g^{-1}h, hg^{-1})f_i\ld_{h^{-1}g}^*\\
&=_{\dt_2}\phi(\sum_{g, i, h}e_{g,h^{-1}g}\otimes f_i\af_{g^{-1}}(u(g,g^{-1})b_h)u(g^{-1},h)u(g^{-1}h, hg^{-1})f_i).
\end{align*}
Where $\dt_1=4\dt|K|(\sum_i|\tilde{T}_i|)$ and $\dt_2=2\dt|K|(\sum_i|\tilde{T}_i|)L$. This shows $pb\subset_{\ep} D$. The proof that $bp\subset_{\ep} D$ is similar.\\
For condition (\ref{i6}), we write 
\beq
1-p=(1-\sum_{g\in T_i, i} e_g^{(i)})+\sum_{g\in T_i\backslash \tilde{T}_i,i}e_g^{(i)}.
\eeq
 For $g\in T_i$, we have $\|e_g^{(i)}-\af_g(f_i)\|<\dt<1$, which implies that the two projections are unitarily equivalent in $A$. Hence for any $\af$-invariant trace $\tau$ and any $g,h\in T_i$, we have $\tau(e_g^{(i)})=\tau(\af_g(f_i))=\tau(f_i)=\tau(e_h^{(i)})$. Therefore
\beq
\tau(\sum_{g\in T_i\backslash \tilde{T}_i, i}e_g^{(i)})=\ep_0\tau(\sum_{g\in T_i, i}e_g^{(i)})\leq \ep_0< \tau(z_0)
\eeq
By Proposition 2.4 of \cite{Osaka2006b} (Although it's stated for real rank zero C*-algebra, but all is needed is that $\V(A)$ is almost divisible, and the same proof works for cocycle actions), we have $ \sum_{g\in T_i\backslash \tilde{T}_i,i}e_g^{(i)}\precsim z_0$ in $A\rtimes_\af G$. Hence 
\beq
1-p=(1-\sum_{g\in T_i,i} e_g^{(i)})+\sum_{g\in T_i\backslash \tilde{T}_i, i}e_g^{(i)}\precsim z_1\oplus z_0\precsim z.
\eeq
\end{proof}

In the following theorem, we need to assume that $A$ has sufficient many projections. Any real rank zero C*-algebra will satisfy the requirement as stated in the theorem.
\begin{thm}\label{scp}
Let $A$ be a unital simple separable C*-algebra with the property that, for any positive element $x\in M_{\infty}(A)$ and any $\ep>0$, there exists a projection $p\in M_{\infty}(A)$ such that $(x-\ep)_+\precsim p\precsim x$. Suppose that $\V(A)$ has strict comparison and is almost divisible. Let $(\af, u)\colon G\cra A$ be an action with the tracial Rokhlin property. Then $\V(A\rtimes_{\af, u} G)$ has strict comparison.
\end{thm}
\begin{proof}
Let $r,s$ be two projections such that $\tau(p)<\tau(q)$, for any $\tau\in \TS(A)$. Let 
\beq
\ep=\min_{\tau\in \TS(A)} \{\tau(r)-\tau(s)\}>0.
\eeq
Let $\dt=\ep/3$. By Lemma \ref{large}, there exists some finite subset $K$ of $G$ and $(K,\dt)$-invariant subsets $T_1, T_2, \dots, T_n$ of $G$, projections $f_1, \dots, f_n\in A$, an embedding $\phi\colon \oplus_i M_{|T_i|}\otimes f_iAf_i\rightarrow A\rtimes_{\af} G$ whose image shall be called $D$ such that
\begin{enumerate}[(1)]
\item \label{i1}there is a $g_i\in T_i$ for each $i$ such that $\phi(e_{g_i,g_i}^{(i)}\otimes a)=\af_{g_i}(a)$, for any $a\in f_iAf_i$.
\item \label{i5}There is a projection $p\in D$ such that
\beq
 pr,ps\subset_{\dt} D\quad \text{and}\quad rp, sp\subset_{\dt} D.
\eeq
\item \label{i6} With the same $p$ as in (\ref{i5}), $\tau(1-p)<\ep/3$, for any $\tau\in \TS(A)$.
\end{enumerate}
Let $x\in D$ be a positive element such that $\|x-prp\|<\dt$. Since $D$ is isomorphic to finite direct sum of matrix algebras over $A$, there is a projection $\tilde{r}\in D$ such that $(x-2\dt)_+\precsim \tilde{r}\precsim (x-\dt)_+$ in $D$. We estimate that $\tilde{r}\precsim (x-\dt)_+\precsim prp\precsim r$. Let $r_0=\tilde{r}\oplus (1-p)$. By Lemma 1.8 of \cite{Phillips2014}, we have
\beq
r\approx (r-\dt)_+ \precsim (prp-\dt)_+\oplus (1-p)\precsim (x-2\dt)_+\oplus (1-p)\precsim r_0.
\eeq
For any $\tau\in \TS(A\rtimes_{\af, u}G)$, we have $\tau(r)> \tau(r_0)-\ep/3$. Similarly, there is a projection $s_0\in D$ such that $s_0\precsim s$ and $\tau(s_0)>\tau(s)-\ep/3$, for any $\tau\in \TS(A\rtimes_{\af, u} G)$. \\
Let $d=\oplus_i d_i$, $e=\oplus_i e_i$ be projections in $\oplus_i M_{|T_i|}\otimes f_iAf_i$ such that $\tilde{r}=\phi(d)$ and $s_0=\phi(e)$. Each $M_{|T_i|} \otimes f_iAf_i$ is simple, hence $d_i\precsim \diag\{e_{g_i,g_i}^{(i)}\otimes f_i, \dots e_{g_i,g_i}^{(i)}\otimes f_i\}$ in $M_{k_i}(M_{|T_i|} \otimes f_iAf_i)$ for some $k_i\in \N$. Let
\beq
f=\diag\{e_{g_1,g_1}^{(1)}\otimes f_1, \dots e_{g_1,g_1}^{(1)}\otimes f_1, \dots, e_{g_n,g_n}^{(n)}\otimes f_n, \dots e_{g_n,g_n}^{(n)}\otimes f_n\}
\eeq
Let $k=\sum_i k_i|T_i|$. Define $\iota\colon \oplus_i M_{|T_i|}\otimes f_iAf_i\rightarrow M_k(A)$ by
 \beq
 (a_1, a_2,\dots, a_n)\mapsto \diag\{a_1, a_2\dots, a_n, 0, 0, \dots\}.
\eeq
Then we have $\iota(d)\precsim f$ in $fM_{k}(A)f$. Let $\tilde{d}\in fM_{k}(A)f$ be a projection such that $\iota(d)\approx \tilde{d}$. Since $\phi(e_{g_i,g_i}^{(i)}\otimes a)=\af_{g_i}(a)$, for any $a\in f_iAf_i$, we can see that for any element $a\in \oplus_i M_{|T_i|}\otimes f_iAf_i$, we have $\phi(a)\approx \iota(a)$ in $M_{\infty}(A\rtimes_{\af, u} G)$. Let $r_1=\tilde{d}\oplus (1-p)\in M_{k+1}(A)$. We have
\beq
r_1\approx \iota(d)\oplus (1-p)\approx \phi(d)\oplus (1-p)=r_0.
\eeq
Similarly, there is a projection $s_1\in M_{l}(A)$ such that $s_1\approx s_0$. Let $\tau$ be any $\af$-invariant trace on $A$, which comes from a trace $\om$ on $A\rtimes_{\af, u} G$ by Corollary \ref{it}. We can compute
\beq
\om(s_1)-\om(r_1)=\om(s_0)-\om(r_0)> \om(s)-\ep/3-(\om(r)+\ep/3)>0
\eeq
By Proposition 2.4 of \cite{Osaka2006b}, we have $r_1\precsim s_1$. Hence
\beq
r\precsim r_0\approx r_1\precsim s_1 \approx s_0\precsim s.
\eeq
\end{proof}

\section{Real and stable rank of the crossed product}
The following lemma says that any single self-adjoint element of the crossed product could be 'tracially' approximated by subalgebras with real rank zero. It is weaker than tracial approximation formulated in (Definition 2.2, \cite{Elliott2008}), but it's good enough to deduce that the crossed product has real rank zero, at least when we know the crossed product has strict comparison for projections.

\begin{lem}\label{seta}
Let $A$ be a simple infinite dimensional C*-algebra with real rank zero and has strict comparison for projections. Let $(\af,u)\colon G\cra A$ be a cocycle action with the tracial Rokhlin property, where $G$ is a countable discrete amenable group. Then for any self-adjoint element $a\in A\rtimes_{\af, u} G$, any $\ep>0$ and any nonzero positive element $z\in A\rtimes_{\af, u} G$, there is a  C*-subalgebra $D$ of $A\rtimes_{\af, u} G$ with real rank zero and a projection $p\in D$ such that:
\begin{enumerate}[(1)]
\item $\|pa-ap\|<\ep.$
\item $pap\in_\ep D.$
\item $1-p\precsim z.$
\end{enumerate}
\end{lem}
\begin{proof}
Let $a$, $\ep$ and $z$ are given as in this lemma. Without loss of generality assume $\|a\|\leq 1$. Choose two nonzero orthogonal projections $z_0$ and $z_1$ in $A$ such that $z_0+z_1\precsim z$. Let $\eta=\Min\{\tau(z_0)\,\vert\,\tau\in \TS(A)\}$. 
 Let
\beq
\dt=\Min\{\ep/4,\,\, 1-\eta/2,\,\, \frac{\eta\ep}{4+(3+\eta)\ep}\}
\eeq
 By Lemma \ref{large}, there exist projections $f_i\in A$, finite subsets $\tilde{T}_i\subset T_i\subset G$ with $\frac{|\tilde{T}_i|}{|T_i|}>1-\dt$, an embedding $\phi\colon \oplus_i M_{|T_i|}\otimes f_iAf_i\rightarrow A\rtimes_\af G$ whose image shall be called $D$, and a projection $q\in D$ such that
\begin{enumerate}
\item Let $e_g^{(i)}=\phi(e_{g,g}^{(i)}\otimes f_i)$, for $g\in T_i$, we have $e_g^{(i)}\in A$.
\item $q=\sum_{g\in \tilde{T_i}, i} e_g^{(i)}$.
\item There exist $d_1$ and $d_2$ in $D$ such that $\|qa-d_1\|<\dt$ and $\|aq-d_2\|<\dt$.
\item $1-q\precsim z_1$.
\end{enumerate}
We can write 
\beq
a=qa+(1-q)aq+(1-q)a(1-q)=_{2\dt} d_1+(1-q)d_2+(1-q)a(1-q).
\eeq
Let $d=d_1+(1-q)d_2$ and $\bar{d}=\frac{d+d^*}{2}$. Then $\bar{d}$ is a self-adjoint element in $D$ such that $\|a-(\bar{d}+(1-q)a(1-q))\|<2\dt$. Let $c=\oplus_i c_i$ be a self-adjoint element in $\oplus_i M_{|T_i|}\otimes f_iAf_i$ such that $\bar{d}=\phi(c)$. \\
Let $N$ be an integer such that $2/\ep\leq N\leq 2/\ep+1$. By our choice of $\dt$, we have 
\beq
(2N+1)|T_i\backslash \tilde{T}_i|\leq (4/\ep+3)\frac{\dt}{1-\dt}|\tilde{T}_i|\leq \eta |\tilde{T}_i|
\eeq
Choose a subset $S_i\subset \tilde{T}_i$ such that $|S_i|=(2N+1)|T_i\backslash\tilde{T}_i|$. Let $r_i=\sum_{g\in S_i} e_{g,g}^{(i)}\otimes f_i$, let $q_i=\sum_{g\in \tilde{T}_i} e_{g,g}^{(i)}\otimes f_i$ and $e_i=\sum_{g\in T_i} e_{g,g}^{(i)}\otimes f_i$. Let $e=\phi(\oplus_i e_i)$. Note that $q=\phi(\oplus_i q_i)$. By Lemma 4.4 of \cite{Osaka2006b}, there is a projection $s_i$ in  $M_{|\tilde{T}_i|}\otimes f_iAf_i$ such that
\beq
e_i-q_i\leq s_i\precsim r_i,\hspace{1cm} \|s_ic_i-c_is_i\|<\frac{1}{N}\leq \ep/2
\eeq
Let $s=\phi(\oplus_i s_i)\geq e-q$. We have $\|s\phi(c)-\phi(c)s\|<\ep/2$. Let $p=e-s\leq q$. For Condition (1), we have:
\begin{align*}
\|pa-ap\|&=\|p(a-((1-q)a(1-q))-(a-(1-q)a(1-q))p\|\\
&\leq 2\dt+\|p\bar{d}-\bar{d}p\|=2\dt+\ep/2\leq \ep.
\end{align*}
Since $p\leq q$, we have $pap=pqaqp\in_\ep D$, this proves Condition (2).\\
Finally, for any $\tau\in T(A\rtimes_\af G)$, we have 
\beq
\tau(s)=\tau(\phi(\oplus_i s_i))\leq \eta\tau(\phi(\oplus_i q_i))\leq \tau(z_0).
\eeq
By Proposition 2.4 of \cite{Osaka2006b}, this shows that $s\precsim z_0$. Hence
\beq
1-p=(1-e)+s\leq (1-q)\oplus s\precsim z_1\oplus z_0\precsim z.
\eeq
\end{proof}

\begin{prp}\label{rr0}
Let $A$ be a unital simple C*-algebra with strict comparison for projections. Suppose for any self-adjoint element $a\in A$, any $\ep>0$ and any nonzero positive element $z\in A$, there is a unital C*-subalgebra $D$ of $A$ with real rank zero and $1_D=p$ such that:
\begin{enumerate}[(1)]
\item $\|pa-ap\|<\ep$,
\item $pap\in_\ep D$,
\item $1-p\precsim z$. 
\end{enumerate}
Then $A$ has real rank zero.
\end{prp}



\begin{proof}
Let $a$ be a self-adjoint element in $A$ and $\ep>0$ be given. Without loss of generality assume $\|a\|=1$. Assume that $a$ is not invertible, otherwise there is nothing to prove. Let $\ep_0=\frac{\ep}{26}$. Let $g\colon [-1,1]\rightarrow [0,1]$ be a continuous function such that
\beq
\text{supp}\,\, g=[-\ep_0,\ep_0] \text{\,\, and\,\,} g(0)=1.
\eeq
Let 
\beq
\ep_1=\Min\{\ep_0, \,\,\frac{1}{4}\Min\{\tau(g(a))\,\vert\,\tau\in \TS(A)\}\}>0.
\eeq
Choose $\dt>0$ such that whenever $a,b$ are normal elements with norm less or equal to 1 and $\|a-b\|<\dt$, then $\|g(a)-g(b)\|\leq \ep_1$, according to Lemma 2.5.11 of \cite{Lin2001a}. We further require that $\dt\leq \ep_1$. Since $A$ has strict comparison, we can find a C*-subalgebra $D$ of $A$ with real rank zero and a projection $p\in D$ such that:
\begin{enumerate}[(1)]
\item $\|pa-ap\|<\dt/2.$
\item there is some self-adjoint element $d\in D$ such that $\|pap-d\|<\dt$.
\item $\tau(1-p)<\dt/2$, for any $\tau\in \TS(A)$.
\end{enumerate}
Replacing $d$ by $pdp$, we may assume that $d\in pDp$. We may also assume that $\|d\|\leq 1$. Since $pDp$ has real rank zero for being a corner of real rank zero C*-algebra, there is a projection $r\in g(d)Dg(d)$ such that $\|rg(d)r-g(d)\|<\dt$. In the following, We shall show that $1-p\precsim r\leq p$ and, for any projection $s\leq r$, we have $\|sa\|<\ep$, $\|as\|<\ep$.\\
The choice of $\dt$ shows that 
\beq
g(a)=_{\ep_1} g(pap+(1-p)a(1-p))=g(pap)+g((1-p)a(1-p))
\eeq
 and $g(pap)=_{\ep_1} g(d)$. Hence for any $\tau\in \TS(A)$, we can compute:
\begin{align*}
\tau(r)&\geq\tau(rg(d)r)\geq\tau(g(d))-\dt\\
&\geq\tau(g(pap))-\ep_1-\ep_1\\
&\geq \tau(g(a))-\tau(g((1-p)a(1-p)))-3\ep_1\\
&\geq \tau(g(a))-\tau(1-p)-3\ep_1>\tau(1-p). 
\end{align*}
Since $A$ has strict comparison, this shows that $1-p\precsim r$.\\
Next, since $r\in g(d)Dg(d)$, we have
\beq
\|rd\|=\lim_{n\rightarrow \infty} \|rg(d)^{1/n}d\|\leq \ep_0.
\eeq
Hence for any projection $s\leq r$, $\|sd\|=\|srd\|\leq\ep_0$. Similarly $\|ds\|\leq \ep_0$. Now combine the fact that $\|pa-pa\|<\dt/2<\ep_0$ and $\|pap-d\|<\ep_0$, we can get
\beq
\|sa\|=\|s(pap)+spa(1-p)+s(1-p)a\|\leq (\|sd\|+\ep_0)+\ep_0\leq \ep.
\eeq
And similarly $\|as\|<\ep$.\\
Now since $1-p\precsim r$, let $v$ be a partial isometry such that $vv^*=1-p$ and $v^*v=s\leq r\leq p$. Using the decomposition $1=(p-s) \oplus s\oplus (1-p)$, we may write $a$ in the matrix form:
\[
a=\left (\begin{array}{ccc}
(p-s)a(p-s) & (p-s)as &(p-s)a(1-p)\\
sa(p-s) & sas & sa(1-p)\\
(1-p)a(p-s) & (1-p)as &(1-p)a(1-p)\\
\end{array}
\right )
\]

Now $(p-s)a(p-s)=_{\ep_0}(p-s)d(p-s)\in (p-s)D(p-s)$. Since $(p-s)D(p-s)$ has real rank zero, there is an invertible self-adjoint element $d_1\in (p-s)D(p-s)$ such that $\|(p-s)d(p-s)-d_1\|<\ep_0$. Hence 
\begin{align*}
a&=_{23\ep_0}
\left (\begin{array}{ccc}
(p-s)d(p-s) & 0 & 0\\
0 & 0 & 0\\
0 & 0 & (1-p)a(1-p)\\
\end{array}
\right )\\
&\quad\\
&=_{2\ep_0}
\left (\begin{array}{ccc}
d_1 & 0 & 0\\
0 & 0 & \ep_0v^*\\
0 & \ep_0v &(1-p)a(1-p)\\
\end{array}
\right )
\end{align*}
The last matrix corresponds to a invertible self-adjoint element $a_0$ in $A$. By our choice of $\ep_0$, we have $\|a-a_0\|<\ep$. 
\end{proof}

Combine Theorem \ref{scp}, Lemma \ref{seta} and Proposition \ref{rr0}, we get the following:
\begin{thm}
Let $A$ be a simple unital C*-algebra with real rank zero and has strict comparison for projections. Let $(\af, u)\colon G\cra A$ be a cocycle action with the tracial Rokhlin property. Then $A\rtimes_\af G$ has real rank zero.
\end{thm}

Now let's turn to the case of stable rank one. We first see that Lemma 5.2 of \cite{Osaka2006b} could be generalized to actions of general amenable groups, because it's proof depends only on Lemma 2.5 and Lemma 2.6 of \cite{Osaka2006b} and some other lemmas unrelated to crossed product. We could use Lemma \ref{large} to replace the first one, and Lemma 2.6 of \cite{Osaka2006b} could be generalized to actions of amenable groups with the same proof. Hence we have:
\begin{lem}\label{sr1l}
Let $A$ be a simple C*-algebra with real rank zero and strict comparison for projections. Let $(\af, u)\colon G\cra A$ be a cocycle action with the tracial Rokhlin property. Then for any nonzero projections $p_1,\dots,p_n\in A\rtimes_{\af, u} G$ and arbitrary elements $a_1,\dots,a_m\in A\rtimes_{\af, u} G$, any $\ep>0$, there exist a unital subalgebra $A_0\subset A\rtimes_{\af, u} G$ which is stably isomorphic to $A$, a projection $p\in A_0$ and subprojections $r_1,\dots,r_n$ of $p$ such that:
\begin{enumerate}[(1)]
\item $pa\in_{\ep} A_0$, $ap\in_{\ep} A_0$.
\item $p_kr_k=_\ep r_k$, for any $k$.
\item $1-p\precsim r_k$, for any $k$.
\end{enumerate} 
\end{lem}

\begin{prp} \label{sr1p}
Let $A$ be a unital simple stably finite C*-algebra with Property (SP). If for any $x\in A$, any $\ep>0$ and any projection $p_1,\dots,p_n$, there is a unital simple subalgebra $D$ with stable rank one and Property (SP), a projection $p\in D$ and subprojections $r_1,\dots,r_n$ of $p$ such that:
\begin{enumerate}[(1)]
\item $pxp\in_{\ep} D$.
\item $r_kp_k=_\ep r_k$.
\item $1-p\precsim r_k$.
\end{enumerate}  
Then $A$ has stable rank one.
\end{prp}
\begin{proof}
Let $x$ be an arbitrary element of $A$ and let $\ep>0$ be given. Without loss of generality assume $\|x\|=1$. Since $A$ is stably finite, every one sided invertible elements is two sided invertible, hence by Theorem 3.3(a) of \cite{Rordam1991}, we may assume that $x$ is a two-sided zero divisor. Since $A$ has property (SP), we can find nonzero projections $e,f$ such that $ex=xf=0$. Let $\ep_0=\ep/11$. We can then find a unital simple subalgebra $D$ with stable rank one and Property (SP), a projection $p\in D$ and sub-projections $e_0, f_0$ of $p$, such that
\beq
e_0e=_{\ep_0} e_0,\hspace{1cm} f_0f=_{\ep_0} f_0
\eeq
Consider $x_0=(1-e_0)x(1-f_0)$. Then
\beq
x_0=_{2\ep_0} (1-e_0e)x(1-ff_0)=x.
\eeq 
Since $D$ is a simple C*-algebra with Property (SP), there is a nonzero projection $r\leq e_0$ and $r\precsim f_0$. Since $D$ has stable rank one, there exists some unitary $u$ such that $uru^*\leq f_0$. Hence $r(x_0u)=(x_0u)r=0$. \\
Next, we shall approximate $x_1=x_0u$ by an invertible element. To this end we find a unital subalgebra $D_1$ of $A$ with stable rank one, a projection $p\in D_1$ and sub-projection $r_1$ of $p$, and an element $d\in D_1$ such that
\beq
px_1p=_{\ep_0} d,\hspace{1cm} r_1r=_{\ep_0} r_1,\hspace{1cm}\text{and}\,1-p\precsim r_1.
\eeq
Choose a partial isometry $v$ such that $vv^*=1-p$ and $v^*v=s\leq r_1\leq p$. According to the decomposition $1= (1-p)\oplus (p-s)\oplus s$, we can write $x_1$ in the matrix form:
\[
x_1=\left (\begin{array}{ccc}
(1-p)x_1(1-p) & (1-p)x_1(p-s) &(1-p)x_1s\\
(p-s)x_1(1-p) & (p-s)x_1(p-s) & (p-s)x_1s \\
sx_1(1-p)       &sx_1(p-s)       &  sx_1s\\
\end{array}
\right )
\]
Since $(p-s)D_1(p-s)$ has stable rank one, there is an invertible element $d_1\in (p-s)D_1(p-s)$ such that 
\beq
d_1=_{\ep_0} (p-s)d(p-s)=_{\ep_0} (p-s)x_1(p-s).
\eeq
We also have $sx_1=sr_1x_1=_{\ep_0} sr_1rx_1=0$, and similarly $x_1s=_{\ep_0}0$. Therefore 
\beq
x_1=_{7\ep_0}
\left (\begin{array}{ccc}
a & b & 0 \\
c & d_1 & 0\\
0 & 0 & 0\\
\end{array}
\right )=_{2\ep_0}
\left (\begin{array}{ccc}

a & b & \ep_0\\
c & d_1& 0\\
\ep_0 & 0& 0\\
\end{array}
\right ).
\eeq
Let's call the last matrix $x_2$, which is invertible. Then 
\beq
\|x-x_2u^*\|\leq \|x-x_0\|+\|(x_0u-x_2)u^*\|<11\ep_0<\ep.
\eeq
Hence $A$ has stable rank one.
\end{proof}

Combine Lemma \ref{sr1l} and Proposition \ref{sr1p} we get:
\begin{thm}\label{sr1}
Let $A$ be a simple unital C*-algebra with real rank zero, stable rank one and has strict comparison for projections. Let $(\af, u)\colon G\cra A$ be a cocycle action with the tracial Rokhlin property. Then $A\rtimes_{\af, u} G$ has stable rank one.
\end{thm}

\bibliographystyle{amsplain}
\bibliography{mybib}
Qingyun Wang, \textsc{Department of Mathematics, University of Toronto}\par\nopagebreak
  \textit{E-mail address}: \texttt{wangqy@math.toronto.edu}
\end{document}
9889iu+